\documentclass[11pt]{amsart}
\usepackage{graphicx}
\usepackage[centertags]{amsmath}
\usepackage{amsfonts}
\usepackage{amssymb}
\usepackage{amsthm}
\usepackage{newlfont}
\usepackage{caption}
\newtheorem{theorem}{Theorem}[section]
\newtheorem{corollary}[theorem]{Corollary}
\newtheorem{lemma}[theorem]{Lemma}
\newtheorem{proposition}[theorem]{Proposition}
\theoremstyle{definition}
\newtheorem{definition}[theorem]{Definition}
\theoremstyle{remark}
\newtheorem{remark}[theorem]{Remark}

\newcommand{\sgn}{\textrm{sgn}}

\title[Three-dimensional Riemannian Manifolds \dots]{Three-dimensional Riemannian Manifolds Associated with Locally Conformal Riemannian Product Manifolds}

\author{Iva Dokuzova}

\address{Iva Dokuzova\\Department of Algebra and Geometry\\
University of Plovdiv Paisii Hilendarski\\ 24 Tzar Asen, 4000 Plovdiv, Bulgaria}
\email{dokuzova@uni-plovdiv.bg}


\begin{document}
\begin{abstract}
	A 3-dimensional Riemannian manifold equipped with a tensor structure of type $(1,1)$, whose fourth power is the identity, is considered. This structure acts as an isometry with respect to the metric. A Riemannian almost product manifold associated with such a manifold is also studied. It turns out, that the almost product manifold belongs to the class of locally conformal Riemannian product manifolds of the Naveira classification. Conditions for the additional structures of the manifolds to be parallel with respect to the Levi-Civita connection of the metric were found. Classes of almost Einstein manifolds and Einstein manifolds are determined and some of their curvature properties are obtained. As examples of these manifolds, a hypersurface is considered.
\end{abstract}

\maketitle


\section{Introduction}
\label{Sec:1}
Investigating the geometric properties of Riemannian manifolds with additional structures enriches differential geometry and leads to more of its applications.
We study the geometry of Riemannian manifolds equipped with an additional tensor structure, which are associated with Riemannian almost product manifolds. In order to convince the reader of the relevance of our research, we will mention some modern papers on this subject refer to the theory of Riemannian almost product manifolds (\cite{alegre, atceken}, \cite{dokuzova}, \cite{gribacheva}--\cite{griba-mek2}, \cite{pitis}--\cite{S-G}).

 A.~Naveira made a classification of the Riemannian almost product manifolds by the properties of the tensor $\nabla P$, where $\nabla$ is the Levi-Civita connection determined by the metric, and $P$ is the almost product structure (\cite{Nav}). The class of locally product manifolds defined by $\nabla P=0$ in this classification is common to all classes, and in this case the structure $P$ is parallel with respect to $\nabla$.

  We introduce a 3-dimensional Riemannian manifold $(M, g, Q)$. Here $g$ is the metric and $Q$ is a tensor field of type $(1,1)$ whose fourth power is the identity. The local coordinates of $Q$ form a special rotation matrix and $Q$ is compatible with $g$, such that an isometry is induced in every tangent space of $M$. Also, we consider an associated manifold $(M, g, P)$, where $P=Q^{2}$ is an almost product structure. Our purpose is to obtain some geometric properties of $(M, g, Q)$ and $(M, g, P)$. We prove that $(M, g, P)$ is a locally conformal product manifold. We consider an associated metric $\tilde{g}$, determined by $g$ and $P$, and relations between curvature quantities of the Riemannian manifold $(M, g, Q)$ and the pseudo-Riemannian manifold $(M, \tilde{g}, Q)$.

The paper is organized as follows. In Section~\ref{Sec:2}, we give some basic facts about $(M, g, Q)$ and $(M, g, P)$. In Section~\ref{Sec:3}, we get the components of the fundamental tensor $F$ on $(M, g, P)$, determined by the covariant derivative of $\tilde{g}$. We prove that $(M, g, P)$ is a locally conformal Riemannian product manifold. We obtain necessary and sufficient conditions for $Q$ and also for $P$ to be parallel structures. In Section~\ref{Sec:4}, we find a relation between the Ricci tensors on $(M, \tilde{g}, Q)$ and $(M, g, Q)$. In case that $(M, \tilde{g}, Q)$ is a locally Euclidean manifold, then the Ricci tensor on $(M, g, Q)$ is expressed by the metrics $g$ and $\tilde{g}$, as well as $(M, g, Q)$ is an almost Einstein manifold. In Section~\ref{Sec:5}, we introduce special bases in the tangent space $T_{p}M$ of $(M, g, Q)$, which are induced by $Q$. We find the Ricci curvatures in the direction of a non-zero vector and its images by $Q$. Also we obtain sectional curvatures of special 2-planes in $T_{p}M$ for an almost Einstein manifold $(M, g, Q)$. In Section~\ref{Sec:6}, we characterize geometrically an example of $(M, g, Q)$ on a 3-dimensional catenoid embedded in a 4-dimensional Euclidean space.

\section{Preliminaries}
\label{Sec:2}
The general representation of a rotation matrix in $\mathbb{R}^{3}$ (rotation is around a coordinate system axis) is
\begin{equation}\label{rot-matrix}
 Q_{i}= \begin{pmatrix}
    \cos\alpha & -\sin \alpha & 0 \\
    \sin\alpha & \cos\alpha & 0 \\
    0 & 0 & 1 \\
  \end{pmatrix},\quad 0<\alpha\leq 2\pi.
\end{equation}
The equation $Q^{4}=id$ applied to $Q_{i}$ has a set of solutions $\alpha=\{\frac{\pi}{2}; \pi; \frac{3\pi}{2}; 2\pi\}$.
Then the matrix \eqref{rot-matrix} has one of the following forms:
\begin{equation*}
\begin{split}
    Q_{1}&=\begin{pmatrix}
      0 & -1 & 0 \\
      1 & 0 & 0 \\
      0 & 0 & 1 \\
    \end{pmatrix},\quad Q_{2}=\begin{pmatrix}
      -1 & 0 & 0 \\
      0 & -1 & 0 \\
      0 & 0 & 1 \\
    \end{pmatrix},\\ Q_{3}&=\begin{pmatrix}
      0 & 1 & 0 \\
      -1 & 0 & 0 \\
      0 & 0 & 1 \\
    \end{pmatrix},\ Q_{4}=\begin{pmatrix}
      1 & 0 & 0 \\
      0 & 1 & 0 \\
      0 & 0 & 1 \\
    \end{pmatrix}.
    \end{split}
\end{equation*}
 It is easy to see that the above matrices form a group with respect to the matrix multiplication. As well as that they are related by $Q_{1}^{2}=Q_{2}$, $Q_{1}^{3}=Q_{3}$ and $Q_{1}^{4}=Q_{4}$. The matrices $Q_{1}$ and $Q_{3}$ satisfy inequalities $Q^{2}\neq id$ and $Q\neq id$.
 Further, we choose to work with structures represented by $Q_{1}$ and $Q_{2}$.

We consider a Riemannian manifold $(M, g, Q)$ and the associated manifold $(M, g, P)$, where $P=Q^{2}$. These manifolds are determined in the following way.

Let $M$ be a $3$-dimensional differentiable manifold equipped with a Riemannian metric $g$.
Let $Q$ be a tensor field on $M$ of type $(1,1)$, whose component matrix is:
\begin{equation}\label{Q}
    (Q_{i}^{j})=\begin{pmatrix}
      0 & -1 & 0 \\
      1 & 0 & 0 \\
      0 & 0 & 1 \\
    \end{pmatrix}.
\end{equation}
Obviously
\begin{equation*}
    Q^{4}= id,\qquad Q^{2}\neq \pm id.
\end{equation*}
Let the structure $Q$ be compatible with $g$ such that
\begin{equation}\label{g-suglasuvane}
     g(Qx, Qy)=g(x,y).
\end{equation}
 Here and anywhere in this work, $x, y, z, t$ will stand for arbitrary elements of the algebra on the smooth vector fields on $M$ or vectors in the tangent space $T_{p}M$ ($p\in M$). The Einstein summation convention is used, the range of the summation indices being always $\{1, 2, 3\}$.

The equalities \eqref{Q} and \eqref{g-suglasuvane} imply that the component matrix of $g$ has the form
\begin{equation}\label{g}
    (g_{ij})=\begin{pmatrix}
      A & 0 & 0 \\
      0 & A & 0 \\
      0 & 0 & B \\
    \end{pmatrix},
\end{equation}
 where $A=A(x^{1}, x^{2}, x^{3})$ and $B=B(x^{1}, x^{2}, x^{3})$ are smooth functions of a point $p(x^{1}, x^{2}, x^{3})$ on $M$. We suppose $A>0$ and $B>0$ in order that the metric $g$ is positive definite.

Bearing in mind \eqref{Q} and \eqref{g-suglasuvane}, we get that the structure $P=Q^{2}$ has a component matrix
\begin{equation}\label{P}
    (P_{i}^{j})=\begin{pmatrix}
      -1 & 0 & 0 \\
       0 & -1 & 0 \\
      0 & 0 & 1 \\
    \end{pmatrix},
\end{equation}
and $g(Px,Py)=g(x,y)$. Therefore, the manifold $(M, g, P)$, where $P = Q^{2}$, is a Riemannian manifold with an almost product structure $P$.

The associated metric $\tilde{g}$ on $(M, g, Q)$ is introduced by
\begin{equation}\label{metricf}
\tilde{g}(x,y)=g(x,Py).\end{equation}
Since the matrices of $g$ and $P$ are determined by \eqref{g} and \eqref{P}, then $\tilde{g}$ has a component matrix
\begin{equation}\label{metric-f}
    (\tilde{g}_{ij})=\begin{pmatrix}
      -A & 0 & 0 \\
      0 & -A & 0 \\
      0 & 0 & B \\
    \end{pmatrix}.
\end{equation}
Due to $A>0$ and $B>0$ the metric $\tilde{g}$ is necessarily indefinite.

With $\nabla$ we denote the Levi-Civita (Riemannian) connection of $g$. The tensor $F$ of type $(0,3)$ and the 1-form $\theta$ are defined by
\begin{equation}\label{F}
  F(x,y,z)=g((\nabla_{x}P)y,z),\quad \theta(z)=g^{ij}F(e_{i},e_{j},z),
\end{equation}
where $\{e_{i}\}$ is a basis of $T_{p}M$ and $g^{ij}$ are the components of the inverse matrix of $(g_{ij})$.

The tensor $F$ has the following properties:
\begin{equation}\label{F-prop}
	F(x,y,z)=F(x,z,y), \quad F(x,Py,Pz)=-F(x,y,z).
\end{equation}
\begin{remark}
The tensor $F$ is fundamental on $(M, g, P)$ and also on $(M, g, Q)$, since it defines basic classes for both manifolds.
\end{remark}
 For the Riemannian almost product manifolds Naveira's classification is valid (\cite{Nav}). This classification is made with respect to the fundamental tensor $F$ and the 1-form $\theta$.
 The almost product manifolds with an integrable structure $P$ are called Riemannian product manifolds. Their subclass of locally conformal Riemannian product manifolds is the largest class, which is closed with respect to the usual conformal transformations of the metric. Moreover, the manifolds of this class have a fundamental tensor $F$ expressed as follows:
   \begin{equation}\label{c1-Q}
\begin{split} F(x,y,z)=&\frac{1}{4pq}\big\{(m\theta(y)+(p-q)\theta(Py))g(x,z)\\&+(m\theta(z)+(p-q)\theta(Pz))g(x,y)\\&-(m\theta(Py)+(p-q)\theta(y))g(x,Pz)\\&-(m\theta(Pz)+(p-q)\theta(z))g(x,Py)\big\}.
\end{split}
\end{equation}
Here $p$ and $q$ are the numbers of the multiplicity of the eigenvalues $1$ and $-1$ of $P$, respectively; $m=p+q$ is the dimension of the manifold.

 The class of Riemannian almost product manifolds with a parallel structure $P$, i.e. locally product manifolds, is determined by
 \begin{equation}\label{c0-Q}
  F(x, y, z)=0.
\end{equation}

\section{Locally conformal Riemannian product manifolds}\label{Sec:3}
In this section, we try to find the locus of $(M, g, Q)$ and also of $(M, g, P)$, bearing in mind the Naveira's classification.
For this purpose we first calculate the components of $F$ and $\theta$.
\begin{lemma}\label{tF}
 The nonzero components $F_{ijk}=F(e_{i}, e_{j}, e_{k})$ of the fundamental tensor $F$ are given by
\begin{equation}\label{nablaf}
 F_{113}=F_{223}=A_{3},\quad  F_{313}=-B_{1},\quad F_{323}=-B_{2},
\end{equation}
where $A_{i}=\dfrac{\partial A}{\partial x^{i}}$, $B_{i}=\dfrac{\partial B}{\partial x^{i}}$.
\end{lemma}
\begin{proof}
If $\Gamma^{k}_{ij}$ are the Riemann-Christoffel symbols of $\nabla$, then
 \begin{equation}\label{2.3}
2\Gamma_{ij}^{s}=g^{as}(\partial_{i}g_{aj}+\partial_{j}g_{ai}-\partial_{a}g_{ij}).
\end{equation}
The inverse matrix of $(g_{ij})$ has the form:
\begin{equation}\label{g-obr}
    (g^{ij})=\begin{pmatrix}
            \frac{1}{A}& 0 & 0 \\
            0& \frac{1}{A}& 0 \\
            0 &0 & \frac{1}{B}\\
            \end{pmatrix}.
  \end{equation}
Using \eqref{g}, \eqref{2.3} and \eqref{g-obr}, we calculate the coefficients $\Gamma^{k}_{ij}$. They are given below:
\begin{equation}\label{gama}
\begin{split}
\Gamma_{11}^{1}&=\dfrac{A_{1}}{2A},\ \Gamma_{11}^{2}=-\dfrac{A_{2}}{2A},\ \Gamma_{11}^{3}=-\dfrac{A_{3}}{2B},\ \Gamma_{12}^{1}=\dfrac{A_{2}}{2A},\ \Gamma_{12}^{2}=\dfrac{A_{1}}{2A},\\ \Gamma_{12}^{3}&=0,\
\Gamma_{22}^{1}=-\dfrac{A_{1}}{2A},\ \Gamma_{22}^{2}=\dfrac{A_{2}}{2A},\ \Gamma_{22}^{3}=-\dfrac{A_{3}}{2B},\
\Gamma_{13}^{1}=\dfrac{A_{3}}{2A},\\ \Gamma_{13}^{2}&=0,\ \Gamma_{13}^{3}=\dfrac{B_{1}}{2B},\
\Gamma_{33}^{1}=-\dfrac{B_{1}}{2A},\ \Gamma_{33}^{2}=-\dfrac{B_{2}}{2A},\ \Gamma_{33}^{3}=\dfrac{B_{3}}{2B},\\
\Gamma_{23}^{1}&=0,\ \Gamma_{23}^{2}=\dfrac{A_{3}}{2A},\ \Gamma_{23}^{3}=\dfrac{B_{2}}{2B}.
\end{split}
\end{equation}
Due to \eqref{metricf} and \eqref{F} the components of $F$ are $F_{ijk}=\nabla_{i}\tilde{g}_{jk}$.
We apply to $\tilde{g}$ the following well-known formula for the covariant derivative of tensors:
\begin{equation}\label{defF}
\nabla_{i}\tilde{g}_{jk}=\partial_{i}\tilde{g}_{jk}-\Gamma_{ij}^{a}\tilde{g}_{ak}-\Gamma_{ik}^{a}\tilde{g}_{aj}.
\end{equation}
Then, with the help of \eqref{metric-f}, \eqref{F-prop} and \eqref{gama} we calculate the components of $F$, given in \eqref{nablaf}.
\end{proof}
\begin{corollary}\label{t-theta}
 The components $\theta_{k}=g^{ij}F(e_{i}, e_{j}, e_{k})$ of the 1-form $\theta$ are expressed by the equalities:
 \begin{equation}\label{theta}
\begin{split}
\theta_{1}=-\frac{B_{1}}{B},\qquad
\theta_{2}=-\frac{B_{2}}{B},\qquad
\theta_{3}=\frac{2A_{3}}{A}.
\end{split}
\end{equation}
\end{corollary}
\begin{proof}
The proof follows from \eqref{F}, \eqref{nablaf} and \eqref{g-obr} by direct computations.
 \end{proof}
Now, bearing in mind the definition of locally conformal Riemannian product manifolds given by \eqref{c1-Q}, we formulate the following
\begin{theorem}\label{tw1}
The manifold $(M, g, P)$ is a locally conformal Riemannian product manifold and the fundamental tensor $F$, determined by \eqref{F}, satisfies the identity
\begin{equation}\label{c1}
\begin{split}
 F(x,y,z)=&\frac{1}{8}\big\{\big(3\theta(y)-\theta(Py)\big)g(x,z)+\big(3\theta(z)-\theta(Pz)\big)g(x,y)\\&-\big(3\theta(Py)-\theta(y)\big)g(x,Pz)-\big(3\theta(Pz)-\theta(z)\big)g(x,Py)\big\}.
\end{split}
\end{equation}
\end{theorem}
\begin{proof}
The equality \eqref{c1} is obtained from \eqref{c1-Q} by substituting $p=1$, $q=2$ and $m=p+q=3$.

We denote by $\widetilde{\theta}_{i}=P_{i}^{a}\theta_{a}$. Then from \eqref{P} and \eqref{theta}, we get
\begin{equation}\label{theta*}
  \widetilde{\theta}_{1}=-\theta_{1},\ \widetilde{\theta}_{2}=-\theta_{2},\ \widetilde{\theta}_{3}=\theta_{3}.
\end{equation}
Using  \eqref{g}, \eqref{metric-f}, \eqref{nablaf}, \eqref{theta} and \eqref{theta*} we obtain
 \begin{equation}\label{usl-w1}
 \begin{split}
  F_{ijk}&=\frac{1}{8}\big(g_{ik}(3\theta_{j}-\tilde{\theta}_{j})+g_{ij}(3\theta_{k}-\tilde{\theta}_{k})\\&-\tilde{g}_{ik}(3\widetilde{\theta}_{j}-\theta_{j})-\tilde{g}_{ij}(3\widetilde{\theta}_{k}-\theta_{k})\big),
  \end{split}
\end{equation}
which is a local form of \eqref{c1}.
\end{proof}

If $F=0$, then the structure $P$ is parallel with respect to $\nabla$. The following necessary and sufficient conditions for $Q$ and also for $P=Q^{2}$ to be parallel structures with respect to $\nabla$ are established.
\begin{theorem}\label{alm-keler}
The manifold $(M, g, P)$ satisfies $\nabla P=0$ if and only if
\begin{align}\label{paralel-P2}
   A=A(x^{1}, x^{2}),\quad B=B(x^{3}).
\end{align}
\end{theorem}
\begin{proof}
Directly from \eqref{c0-Q} and \eqref{nablaf} it follows \begin{align}\label{paralel-P}
   A_{3}=0,\quad B_{1}=0,\quad B_{2}=0.
\end{align}
Consequently we get \eqref{paralel-P2}.
\end{proof}
\begin{theorem} The structure $Q$ on $(M, g, Q)$ is parallel if and only if the structure $P$ on $(M, g, P)$ is parallel.
\end{theorem}
\begin{proof}
Applying \eqref{Q} and \eqref{gama} into \begin{equation*}
\nabla_{i}Q^{k}_{j}=\partial_{i}Q^{k}_{j}+\Gamma_{ia}^{k}Q^{a}_{j}-\Gamma_{ij}^{a}Q^{k}_{a},
\end{equation*} we calculate
\begin{equation*}
\begin{split}
  \nabla_{1}Q^{3}_{1}&=-\frac{A_{3}}{2B},\quad \nabla_{1}Q^{3}_{2}=-\frac{A_{3}}{2B},\quad \nabla_{2}Q^{3}_{1}=\frac{A_{3}}{2B},\\ \nabla_{2}Q^{3}_{2}&=\frac{A_{3}}{2B},\quad
  \nabla_{3}Q^{1}_{3}=\frac{B_{2}-B_{1}}{2A},\quad \nabla_{3}Q^{2}_{3}=-\frac{B_{1}+B_{2}}{2A},\\ \nabla_{1}Q^{1}_{3}&=\frac{A_{3}}{2A},\quad \nabla_{1}Q^{2}_{3}=\frac{A_{3}}{2A},\quad
\nabla_{3}Q^{2}_{1}=\frac{A_{3}}{2A},\\ \nabla_{3}Q^{3}_{1}&=-\frac{B_{1}+B_{2}}{2B},\quad \nabla_{2}Q^{1}_{3}=-\frac{A_{3}}{2A},\quad \nabla_{2}Q^{2}_{3}=\frac{A_{3}}{2A}.
  \end{split}
\end{equation*}
The remaining components of $\nabla Q$ are equal to zero.
Therefore the condition $\nabla Q=0$ is valid if and only if the functions $A$ and $B$ satisfy \eqref{paralel-P}.
\end{proof}

\section{Almost Einstein manifolds}\label{Sec:4}
In this section, we investigate some curvature quantities of $(M, g, Q)$, corresponding to the metric $g$ and to the associated metric $\tilde{g}$, defined by \eqref{metricf}. We also determine classes of Einstein and almost Einstein manifolds.

The curvature tensor $R$ of $\nabla$ is defined by \begin{equation}\label{R-def}R(x, y)z=\nabla_{x}\nabla_{y}z-\nabla_{y}\nabla_{x}z-\nabla_{[x,y]}z.\end{equation}
Also we consider the tensor of type $(0, 4)$, associated with $R$, defined by
\begin{equation}\label{R-4}
    R(x, y, z, t)=g(R(x, y)z,t).
\end{equation}

The Ricci tensor $\rho$ and the scalar curvature $\tau$, with respect to $g$, are given by the well-known formulas:
\begin{equation}\label{def-rho}
    \rho(y,z)=g^{ij}R(e_{i}, y, z, e_{j}),\quad
    \tau=g^{ij}\rho(e_{i}, e_{j}).
\end{equation}

A Riemannian manifold is said to be Einstein if its Ricci tensor $\rho$ is a multiple of the metric tensor $g$ and a smooth function on $M$, i.e.
\begin{equation}\label{E}\rho(x, y) = \alpha g(x, y).\end{equation}

In \cite{Yano}, for locally decomposable Riemannian manifolds is defined a class of almost Einstein manifolds.
For the considered in our paper manifolds, we suggest the following
\begin{definition}\label{defAE}
A Riemannian manifold $(M, g, Q)$ is called
almost Einstein if its Ricci tensor $\rho$ and the metrics $g$ and $\tilde{g}$ satisfy
\begin{equation}\label{AE}
\rho(x, y) = \alpha g(x, y) + \beta \tilde{g}(x, y),\end{equation} where $\alpha$ and $\beta$ are smooth functions on $M$.
\end{definition}

Let $\tilde{\Gamma}$ be the Christoffel symbols of $\tilde{g}$ and $\tilde{\nabla}$ be the Levi-Civita connection of $\tilde{g}$. Let $\tilde{R}$
be the curvature tensor of $\tilde{\nabla}$. The Ricci tensor $\tilde{\rho}$ and the scalar curvature $\tilde{\tau}$ with respect to $\tilde{g}$ are given by
\begin{equation}\label{def-rho2}
    \tilde{\rho}(y,z)=\tilde{g}^{ij}\tilde{R}(e_{i}, y, z, e_{j}),\quad
    \tilde{\tau}=\tilde{g}^{ij}\tilde{\rho}(e_{i}, e_{j}).
\end{equation}
Here $\tilde{g}^{ij}$ are the components of the inverse matrix of $(\tilde{g}_{ij})$.
Let us denote
\begin{equation}\label{def-rho*}
    \tau^{*}=\tilde{g}^{ij}\rho(e_{i}, e_{j}),\quad \tilde{\tau}^{*}=g^{ij}\tilde{\rho}(e_{i}, e_{j}).
\end{equation}
Now we establish the following
\begin{theorem}\label{connR-R}
 Let $\tilde{g}$ be the associated metric on $(M, g, Q)$ defined by \eqref{metricf}. For the Ricci tensors  $\rho$ and $\tilde{\rho}$ and for the scalar quantities $\tau$, $\tau^{*}$, $\tilde{\tau}$ and $\tilde{\tau}^{*}$ the following relation is valid:
 \begin{equation}\label{con-AE}
 \begin{split}
     \tilde{\rho}(x,y) = & \rho(x,y)+\frac{1}{8}(3\tilde{\tau}^{*}+\tilde{\tau}-3\tau-\tau^{*})g(x,y)\\+&\frac{1}{8}(3\tilde{\tau}+\tilde{\tau}^{*}-3\tau^{*}-\tau)\tilde{g}(x,y).
\end{split}\end{equation}
\end{theorem}
\begin{proof}
From \eqref{defF}, applying the Christoffel formulas \eqref{2.3} to $\Gamma$ and also to $\tilde{\Gamma}$, we obtain
\begin{equation*}
 \tilde{\Gamma}^{k}_{ij}=\Gamma^{k}_{ij}+\frac{1}{2}\tilde{g}^{ks}(\nabla_{i}\tilde{g}_{js}+\nabla_{j}\tilde{g}_{is}-\nabla_{s}\tilde{g}_{ij}).
 \end{equation*}
 Substituting \eqref{usl-w1} into the above equality, we get
  \begin{equation*}
 \tilde{\Gamma}^{k}_{ij}=\Gamma^{k}_{ij}+\frac{1}{8}\tilde{g}^{ks}\big(g_{ij}(3\theta_{s}-\widetilde{\theta_{s}})-\tilde{g}_{ij}(3\theta_{s}-\widetilde{\theta_{s}})\big).
 \end{equation*}
Now we calculate the components of the tensor $\textrm{T}=\tilde{\Gamma}-\Gamma$ of the affine deformation. They are as follows:
 \begin{equation}\label{torsion}
 T^{k}_{ij}=\frac{1}{8}\big(g_{ij}(3\widetilde{\theta}^{k}-\theta^{k})-\tilde{g}_{ij}(3\widetilde{\theta}^{k}-\theta^{k})\big).
 \end{equation}
Here we used the equalities
\begin{equation}\label{g-tilde-g}
  \tilde{g}^{ka}\theta_{a}=\widetilde{\theta}^{k},\quad \tilde{g}_{ka}\theta^{a}=\widetilde{\theta}_{k},
\end{equation}
which follow from \eqref{metricf} and \eqref{F}.

For the components of the curvature tensors $\tilde{R}$ and $R$, it is well-known the relation
$$ \tilde{R}^{k}_{ijs} = R^{k}_{ijs} + \nabla_{j}T^{k}_{is}-\nabla_{s}T^{k}_{ij}
+T^{a}_{is}T^{k}_{aj}-T^{a}_{ij}T^{k}_{as}.$$
Then, taking into account \eqref{usl-w1}, \eqref{torsion} and \eqref{g-tilde-g}, we
calculate
\begin{equation*}
\begin{split}
     \tilde{R}^{k}_{ijs} = R^{k}_{ijs} &+ \frac{1}{8}g_{is}\big(3\nabla_{j}\widetilde{\theta}^{k}-\nabla_{j}\theta^{k}+\frac{1}{8}(3\widetilde{\theta}_{j}-\theta_{j})(3\widetilde{\theta}^{k}-\theta^{k})\big)\\&-\frac{1}{8}g_{ij}\big(3\nabla_{s}\widetilde{\theta}^{k}-\nabla_{s}\theta^{k}+\frac{1}{8}(3\widetilde{\theta}_{s}-\theta_{s})(3\widetilde{\theta}^{k}-\theta^{k})\big)
 \\&-\frac{1}{8}\tilde{g}_{is}\big(3\nabla_{j}\theta^{k}-\nabla_{j}\widetilde{\theta}^{k}+\frac{1}{8}(3\theta_{j}-\widetilde{\theta}_{j})(3\widetilde{\theta}^{k}-\theta^{k})\big)\\&+\frac{1}{8}\tilde{g}_{ij}\big(3\nabla_{s}\theta^{k}-\nabla_{s}\widetilde{\theta}^{k}+\frac{1}{8}(3\theta_{s}-\widetilde{\theta}_{s})(3\widetilde{\theta}^{k}-\theta^{k})\big).
\end{split}
\end{equation*}
By contracting $k=s$ in the latter equality, and having in mind   \eqref{usl-w1}, \eqref{def-rho}, \eqref{def-rho2}, \eqref{def-rho*} and \eqref{g-tilde-g}, we find
\begin{equation}\label{tilde-S}
\begin{split}
     \tilde{\rho}_{ij} = \rho_{ij}+\frac{1}{8}g_{ij}(-3\nabla_{s}\widetilde{\theta}^{s}+\nabla_{s}\theta^{s})+\frac{1}{8}\tilde{g}_{ij}(3\nabla_{s}\theta^{s}-\nabla_{s}\widetilde{\theta}_{s}).
\end{split}
\end{equation}
We note that the trace of $P$ is equal to $-1$ and due to \eqref{def-rho}, \eqref{def-rho2}, \eqref{def-rho*} and \eqref{tilde-S} we obtain the system of equations
\begin{equation*}
\begin{split}
     &\tilde{\tau} = \tau^{*}+\frac{(-1)}{8}(-3\nabla_{s}\widetilde{\theta}^{s}+\nabla_{s}\theta^{s})+\frac{3}{8}(3\nabla_{s}\theta^{s}-\nabla_{s}\widetilde{\theta}_{s}),\\
     &\tilde{\tau}^{*} = \tau+\frac{3}{8}(-3\nabla_{s}\widetilde{\theta}^{s}+\nabla_{s}\theta^{s})+\frac{(-1)}{8}(3\nabla_{s}\theta^{s}-\nabla_{s}\widetilde{\theta}_{s}).
\end{split}
\end{equation*}
Then from \eqref{tilde-S} we get
\begin{equation*}
     \tilde{\rho}_{ij} = \rho_{ij}+\frac{1}{8}(3\tilde{\tau}^{*}+\tilde{\tau}-3\tau-\tau^{*})g_{ij}+\frac{1}{8}(3\tilde{\tau}+\tilde{\tau}^{*}-3\tau^{*}-\tau)\tilde{g}_{ij},
\end{equation*}
which is a local form of \eqref{con-AE}.
\end{proof}
We immediately state the following propositions.
\begin{proposition}\label{p4.3}
   Let the Levi-Civita connection $\tilde{\nabla}$ of $\tilde{g}$ be a locally flat connection on $(M, g, Q)$. Then
 $(M, g, Q)$ is an almost Einstein manifold, and the Ricci tensor $\rho$ has the form
 \begin{equation}\label{rho-AE}
\rho(x, y) = \frac{3\tau+\tau^{*}}{8}g(x, y)+\frac{3\tau^{*}+\tau}{8}\tilde{g}(x,y).
\end{equation}
\end{proposition}
\begin{proof}
If $\tilde{\nabla}$ is a locally flat connection, then $\tilde{R}=0$. From \eqref{def-rho2} and \eqref{def-rho*} it follows $\tilde{\rho}=0$ and $\tilde{\tau}=\tilde{\tau}^{*}=0$. Hence \eqref{con-AE} implies \eqref{rho-AE}. Therefore, according to \eqref{AE}, we have that $(M, g, Q)$ is an almost Einstein manifold.
\end{proof}
\begin{corollary}
If the Ricci tensor $\rho$ on $(M, g, Q)$ is determined by \eqref{rho-AE}, then $(M, g, Q)$ is an Einstein manifold if and only if
\begin{equation}\label{tau-E}
  \tau^{*}=-\frac{\tau}{3}.
\end{equation}
\end{corollary}
\begin{proof}
Comparing \eqref{E} and \eqref{rho-AE} we get \eqref{tau-E}. Vice versa, by subtituting \eqref{tau-E} into \eqref{rho-AE} we obtain
\begin{equation}\label{E2}
\rho(x, y) = \frac{\tau}{3}g(x, y),
\end{equation}
i.e. $(M, g, Q)$ is an Einstein manifold.
\end{proof}
\begin{proposition}\label{p4.4}
 Let the Levi-Civita connection $\nabla$ of $g$ be a locally flat connection on the manifold $(M, \tilde{g}, Q)$. Then $(M, \tilde{g}, Q)$ is an almost Einstein manifold and the Ricci tensor $\tilde{\rho}$ has the form
\begin{equation*} \tilde{\rho}(x, y) = \frac{3\tilde{\tau}^{*}+\tilde{\tau}}{8}g(x,y)+\frac{3\tilde{\tau}+\tilde{\tau}^{*}}{8}\tilde{g}(x, y).\end{equation*}
\end{proposition}

\section{Curvature properties}\label{Sec:5}
In this section we consider special bases of the tangent space $T_{p}M$ of $(M, g, Q)$. We find a relation between Ricci curvatures of $(M, g, Q)$ and $(M, \tilde{g}, Q)$ in the direction of a nonzero vector. We also obtain sectional curvatures of 2-planes formed by vectors of such bases in case when $(M, g, Q)$ is an almost Einstein manifold. For this purpose we recall definitions of these curvatures.

The Ricci curvature, with respect to $g$, in the direction of a nonzero vector $x$ is the value \begin{equation}\label{Ricicurv}
    r(x)=\frac{\rho(x,x)}{g(x,x)}.
\end{equation}

The sectional curvature of a non-degenerate $2$-plane $\{x, y\}$ spanned by the vectors $x, y \in T_{p}M$ is the value
\begin{equation}\label{3.3}
    k(x,y)=\frac{R(x, y, x, y)}{g(x, x)g(y, y)-g^{2}(x, y)}.
\end{equation}
\begin{definition}
Every basis of the type $\{x, Qx, Q^{2}x\}$, $\{x, Q^{2}x, Q^{3}x\}$,\newline $\{x, Qx, Q^{3}x\}$ and $\{Qx, Q^{2}x, Q^{3}x\}$ of $T_{p}M$ $(p\in M)$ is called a $Q$-\textit{basis}. In this case we say that \textit{the vector $x$ induces a $Q$-basis of} $T_{p}M$.
\end{definition}
\begin{lemma}
The vector $x(x^{1}, x^{2}, x^{3})$ induces a $Q$-basis of $T_{p}M$ if and only if $x^{3}\big((x^{1})^{2}+(x^{2})^{2}\big)\neq 0$.
\end{lemma}
\begin{proof}
By \eqref{Q}, we obtain that the images of a vector $x(x^{1}, x^{2}, x^{3})$ have components $Qx(-x^{2}, x^{1}, x^{3})$, $Q^{2}x(-x^{1}, -x^{2}, x^{3})$, $Q^{3}x(x^{2}, -x^{1}, x^{3})$. Thus, we note that $Q$ is a rotation matrix in the coordinate plane $Ox^{1}x^{2}$.

We consider the system of vectors $\{x, Qx, Q^{2}x\}$. The triple product of these vectors is $$\left|
                                  \begin{array}{ccc}
                                    x^{1} & x^{2} & x^{3}\\
                                    -x^{2} & x^{1} & x^{3} \\
                                    -x^{1} & -x^{2} & x^{3}\\
                                  \end{array}
                                \right|=2x^{3}\big((x^{1})^{2}+(x^{2})^{2}\big).$$
Therefore the inequality $x^{3}\big((x^{1})^{2}+(x^{2})^{2}\big)\neq 0$ is a necessary and sufficient condition for $\{x, Qx, Q^{2}x\}$ to be a basis of $T_{p}M$.
The same is the result for the systems $\{x, Qx, Q^{3}x\}$, $\{x, Q^{2}x, Q^{3}x\}$ and $\{Qx, Q^{2}x, Q^{3}x\}$.

So, if $\{x, Qx, Q^{2}x\}$ is a basis of $T_{p}M$, then $\{x, Qx, Q^{3}x\}$, $\{x, Q^{2}x, Q^{3}x\}$ and $\{Qx, Q^{2}x, Q^{3}x\}$ are also bases of $T_{p}M$.
\end{proof}
\begin{lemma}\label{angles}
If a vector $x$ induces a $Q$-basis of $T_{p}M$, $\varphi$ is the angle between $x$ and $Qx$ with respect to $g$, $\psi$ is the angle between $x$ and $Q^{2}x$ with respect to $g$, then
   \begin{equation}\label{varphi}
 \angle(x, Q^{3}x)=\varphi, \quad \varphi\in(0, \frac{\pi}{2}\big),\quad \varphi<\psi.
   \end{equation}
   \end{lemma}
   \begin{proof}
   Let $x(x^{1}, x^{2}, x^{3})$ induce a $Q$-basis of $T_{p}M$. Using \eqref{Q}, \eqref{g-suglasuvane} and \eqref{g} we find
   \begin{equation*}
   \begin{split}
   g(x, x)&=g(Qx, Qx)=g(Q^{2}x, Q^{2}x)\\&=g(Q^{3}x, Q^{3}x)
= A\big((x^{1})^{2}+(x^{2})^{2}\big)+B(x^{3})^{2},\\
  g(x, Qx)&=g(x, Q^{3}x)=B(x^{3})^{2},\quad
  g(x,Q^{2}x)=-A\big((x^{1})^{2}+(x^{2}\big)^{2})+B(x^{3})^{2}.
   \end{split}
   \end{equation*}
Having in mind the above equalities and the well-known formula $$\cos\angle(x,y)=\frac{g(x,y)}{\sqrt{g(x,x)}\sqrt{g(y,y)}},$$ we calculate
\begin{equation*}
   \begin{split}
   \cos\varphi =\frac{B(x^{3})^{2}}{A\big((x^{1})^{2}+(x^{2})^{2}\big)+B(x^{3})^{2}},\
   \cos\psi =\frac{-A\big((x^{1})^{2}+(x^{2})^{2}\big)+B(x^{3})^{2}}{A\big((x^{1})^{2}+(x^{2})^{2}\big)+B(x^{3})^{2}}.
   \end{split}
   \end{equation*}
Since $A>0$ and $B>0$ we get $\cos\varphi>0$ and $ \cos\psi<\cos\varphi$, i.e. $0<\varphi<\frac{\pi}{2}$ and $\psi>\varphi$.   \end{proof}
   \begin{remark} The Lemma~\ref{angles} shows that an orthogonal $Q$-basis of $T_{p}M$ does not exist.
   \end{remark}
Due to Theorem~\ref{connR-R}, Proposition~\ref{p4.3} and Proposition~\ref{p4.4} we establish the following statements.
\begin{theorem}
Let a vector $x$ induce a $Q$-basis and let $\psi$ be the angle between $x$ and $Q^{2}x$. If $r$ and $\tilde{r}$ are the Ricci curvatures in the direction of $x$ with respect to the metrics of $(M, g, Q)$ and $(M, \tilde{g}, Q)$, then
 \begin{equation}\label{con-r}
     \tilde{r}(x) = \frac{1}{\cos\psi}r(x)+\frac{1}{8\cos\psi}(3\tilde{\tau}^{*}+\tilde{\tau}-3\tau-\tau^{*})+\frac{1}{8}(3\tilde{\tau}+\tilde{\tau}^{*}-3\tau^{*}-\tau),
\end{equation}
where $\psi\neq\frac{\pi}{2}$.
\end{theorem}
\begin{proof}
We apply \eqref{Ricicurv} to $r$ and also to $\tilde{r}$, and bearing in mind \eqref{con-AE} and
\begin{equation}\label{g-cos2}
\tilde{g}(x, x)=g(x, Q^{2}x)=g(x, x)\cos\psi,
\end{equation}
 we get \eqref{con-r}.
\end{proof}
\begin{proposition}\label{th3.4}
   Let the Levi-Civita connection $\tilde{\nabla}$ of $\tilde{g}$ be a locally flat connection on $(M, g, Q)$. If a vector $x$ induces a $Q$-basis, then the Ricci curvatures in the direction of the basis vectors are
\begin{equation}\label{ricc}
     r(x)=r(Qx)=r(Q^{2}x)=r(Q^{3}x)=\frac{\cos\psi}{8}(3\tau^{*}+\tau)+\frac{1}{8}(3\tau+\tau^{*}),
\end{equation}
where $\psi=\angle(x, Q^{2}x)$.
\end{proposition}
\begin{proof}
Since $\rho$ is given by \eqref{rho-AE}, using \eqref{g-suglasuvane} and \eqref{metricf}, we obtain
\begin{equation}\label{rhoL2}
\begin{split}
    \rho(x, x)&=\rho(Qx, Qx)=\rho(Q^{2}x, Q^{2}x)\\&=\rho(Q^{3}x, Q^{3}x)=\frac{1}{8}(3\tau+\tau^{*})g(x,x)+\frac{1}{8}(3\tau^{*}+\tau)\tilde{g}(x,x).
    \end{split}
\end{equation}
Let a vector $x$ induce a $Q$-basis. Hence equalities \eqref{Ricicurv}, \eqref{g-cos2} and \eqref{rhoL2} imply \eqref{ricc}.
\end{proof}

\begin{corollary}
If the Ricci tensor $\rho$ on $(M, g, Q)$ is determined by \eqref{E2}, and a vector $x$ induces a $Q$-basis, then the Ricci curvatures in the direction of the basis vectors are
\begin{equation*}
     r(x)=r(Qx)=r(Q^{2}x)=r(Q^{3}x)=\frac{\tau}{3}.
\end{equation*}
\end{corollary}
 The proof follows directly by susbtituting \eqref{tau-E} into \eqref{ricc}.
\begin{proposition}
 Let the Levi-Civita connection $\nabla$ of $g$ be a locally flat connection on $(M, \tilde{g}, Q)$.
If a vector $x$ induces a $Q$-basis, then the Ricci curvatures in the direction of the basis vectors are
\begin{equation*} \tilde{r}(x)=\tilde{r}(Qx)=\tilde{r}(Q^{2}x)=\tilde{r}(Q^{3}x)=\frac{1}{8\cos\psi}(3\widetilde{\tau}^{*}+\widetilde{\tau})+\frac{1}{8}(3\widetilde{\tau}+\widetilde{\tau}^{*}),\quad \psi\neq\frac{\pi}{2}.
\end{equation*}
\end{proposition}
 The existence of a class of almost Einstein manifolds is confirmed by Proposition~\ref{p4.3}. In the next theorem, we express the curvature tensor $R$ on an almost Einstein manifold $(M, g, Q)$ by both structures $g$ and $Q$.
\begin{theorem}\label{th2.4}
Let the Ricci tensor $\rho$ on $(M, g, Q)$ be determined by \eqref{rho-AE}. Then the curvature tensor $R$ has the form
 \begin{equation}\label{fR}
R = \frac{\tau+\tau^{*}}{4}\pi_{1}+\frac{3\tau^{*}+\tau}{8}\pi_{2},
\end{equation}
where \begin{align}\label{pi-pi} \nonumber \pi_{1}(x, y, z, u )&=g(y, z)g(x, u) - g(x, z)g(y, u),\\
\pi_{2}(x, y, z, u)&=\mbox{}g(y, z)\tilde{g}(x, u)+g(x, u)\tilde{g}(y, z)\\\nonumber &-g(x, z)\tilde{g}(y, u)-g(y, u)\tilde{g}(x, z).
\end{align}
\end{theorem}
\begin{proof}
It is known that the curvature tensor $R$ on a 3-dimensional Riemannian manifold is completely determined by the Ricci tensor $\rho$ an the metric $g$ by \cite{rashevski}:
\begin{equation*}
\begin{split}
  R(x,y,z,u)&=-g(x, z)\rho(y, u)-g(y, u)\rho(x, z)+g(y, z)\rho(x, u)\\&+g(x, u)\rho(y, z)+\frac{\tau}{2}\big(g(x, z)g(y, u)-g(y, z)g(x, u)\big).
  \end{split}
\end{equation*}
By substituting \eqref{rho-AE} into the above equality we obtain \eqref{fR} and \eqref{pi-pi}.
\end{proof}
In the next corollary we suppose that $(M, g, Q)$ is an Einstein manifold.
\begin{corollary}
If the Ricci tensor on $(M, g, Q)$ is determined by \eqref{E2}, then the curvature tensor $R$ is expressed by
\begin{equation*}
R = \frac{\tau}{6}\pi_{1}.
\end{equation*}
\end{corollary}
\begin{proof}
The above equality follows directly from \eqref{tau-E}, \eqref{fR} and \eqref{pi-pi}.
\end{proof}

Let a vector $x$ induce a $Q$-basis. There are determined six 2-planes $\{x, Qx\}$, $\{x, Q^{2}x\}$, $\{x, Q^{3}x\}$, $\{Qx, Q^{2}x\}$, $\{Qx, Q^{3}x\}$ and $\{Q^{2}x, Q^{3}x\}$ in $T_{p}M$.
In the next theorem we establish the relations among the sectional curvatures of the 2-planes generated by the vectors $\{x, Qx, Q^{2}x, Q^{3}x\}$, the angles $\varphi$ and $\psi$, the scalar quantities $\tau$ and $\tau^{*}$.
\begin{theorem}\label{t2.5}
Let the Ricci tensor $\rho$ on $(M, g, Q)$ have the form \eqref{rho-AE} and let a vector $x$ induce a $Q$-basis. Then the sectional curvatures of the 2-planes, determined by the basis vectors, are
\begin{equation}\label{obmu}
\begin{split}
    k(x,Qx)&=k(Qx,Q^{2}x)= k(x,Q^{3}x)=k(Q^{2}x,Q^{3}x)\\&=-\frac{\tau+\tau^{*}}{4}+(\cos^{2}\varphi- \cos\psi)\frac{\tau+3\tau^{*}}{4(1-\cos^{2}\varphi)},\\
    k(x,Q^{2}x)&= k(Qx,Q^{3}x)=-\frac{\tau+\tau^{*}}{4},
    \end{split}
\end{equation}
where $\varphi=\angle(x, Qx)$, $\psi=\angle(x, Q^{2}x)$.
\end{theorem}
\begin{proof}
Let a vector $x$ induce a $Q$-basis. The conditions \eqref{g-suglasuvane} and \eqref{varphi} imply \begin{equation}\label{g-cos}
\begin{split}
&g(x, Qx)=g(Qx, Q^{2}x)=g(Q^{2}x, Q^{3}x)=g(x, Q^{3}x)=g(x, x)\cos\varphi,\\
&g(x, Q^{2}x)=g(Qx, Q^{3}x)=g(x, x)\cos\psi.
\end{split}
\end{equation}
 Hence, from \eqref{g-suglasuvane}, \eqref{metricf}, \eqref{varphi}, \eqref{g-cos2} and \eqref{g-cos}, we find
 \begin{equation*}
  \begin{split}
   &\tilde{g}(x, x)=\tilde{g}(Qx, Qx)=\tilde{g}(Q^{2}x, Q^{2}x)=\tilde{g}(Q^{3}x,Q^{3}x)=g(x, x)\cos\psi,\\ &\tilde{g}(x, Q^{2}x)=\tilde{g}(Qx, Q^{3}x)=g(x,x),\\&
   \tilde{g}(x, Qx)=\tilde{g}(x, Q^{3}x)=\tilde{g}(Qx, Q^{2}x)=\tilde{g}(Q^{2}x, Q^{3}x)=g(x,x)\cos\varphi.
  \end{split}
 \end{equation*}
Then applying \eqref{fR}, \eqref{pi-pi} and \eqref{g-cos} in \eqref{3.3}, we obtain \eqref{obmu}.
\end{proof}
\begin{corollary}
If the Ricci tensor on $(M, g, Q)$ has the form \eqref{E2}, then the sectional curvatures of the 2-planes, determined by the basis vectors, are
\begin{equation*}
\begin{split}
    k(x,Qx)&=k(Qx,Q^{2}x)= k(x,Q^{3}x)=k(Q^{2}x,Q^{3}x)\\&=
    k(x,Q^{2}x)= k(Qx,Q^{3}x)=-\frac{\tau}{6}.
    \end{split}
\end{equation*}
\end{corollary}
\begin{proof}
These equalities are obtained by substituting \eqref{tau-E} into \eqref{obmu}.
\end{proof}

\section{Hypersurface as an example of the manifolds}\label{Sec:6}

Let $\mathbb{E}^{4}$ be the four-dimensional Euclidean space $(\mathbb{R}^{4}, <\cdot,\cdot>)$ with the usual inner product $$<X,Y>=X^{1}Y^{1}+X^{2}Y^{2}+X^{3}Y^{3}+X^{4}Y^{4},$$
where $X(X^{1}, X^{2}, X^{3}, X^{4})$ and $Y(Y^{1}, Y^{2}, Y^{3}, Y^{4})$ are vectors in $\mathbb{R}^{4}$.

We consider a 3-dimensional catenoid $S\in\mathbb{E}^{4}$ with parametrization \begin{equation}\label{catenoid}
S:r(\cosh u\cos v, \cosh u\sin v,u\cos w,u\sin w), \ u\neq 0, \ v, w\in [0, 2\pi).\end{equation}
We denote by $\partial_{1}=\frac{\partial r}{\partial u}$, $\partial_{2}=\frac{\partial r}{\partial v}$, $\partial_{3}=\frac{\partial r}{\partial w}$ the local basis vectors. The coefficients of the first fundamental form are as follows: \begin{equation*}
\begin{split}
g_{11}&=<\partial_{1},\partial_{1}>=\cosh^{2}u,\ g_{22}=<\partial_{2},\partial_{2}>=\cosh^{2}u,\\& g_{33}=<\partial_{3},\partial_{3}>=u^{2},\ g_{12}=g_{23}=g_{13}=<\partial_{i},\partial_{j}>=0.\end{split}\end{equation*}
Therefore the manifold $(S, g, Q)$, where the metric $g$ is determined by the above equalities and $Q$ is determined by \eqref{Q}, is of the considered type $(M, g, Q)$.

We use $e_{i}=\frac{1}{\sqrt{<\partial_{i}, \partial_{i}>}}\partial_{i}$ and find orthonormal basis vectors
\begin{equation*}
e_{1}=\frac{1}{\cosh u}\partial_{1},\ e_{2}=\frac{1}{\cosh u}\partial_{2},\ e_{3}=\frac{\varepsilon}{u}\partial_{3},\ \varepsilon=\sgn(u),\end{equation*}
i.e. the components of the metric $g$ with respect to $\{e_i\}$ are:
\begin{equation}\label{g-e}
\begin{array}{l}
g(e_{1}, e_{1})= g(e_{2}, e_{2})=g(e_{3}, e_{3})=1,\\
g(e_{1}, e_{2})= g(e_{2}, e_{3})=g(e_{1}, e_{3})=0.
\end{array}
\end{equation}
Now we calculate the commutators
\begin{equation}\label{skobki}
  [e_{1}, e_{2}]=-\frac{\sinh u}{\cosh^{2}u}e_{2},\ [e_{1}, e_{3}]=-\frac{1}{u\cosh u}e_{3},\ [e_{2}, e_{3}]=0.
\end{equation}
The well-known Koszul formula implies
\begin{equation*}
    2g(\nabla_{e_{i}}e_{j}, e_{k})=g([e_{i}, e_{j}],e_{k})+g([e_{k}, e_{i}],e_{j})+g([e_{k}, e_{j}],e_{i})
\end{equation*}
 and, using \eqref{g-e} and \eqref{skobki}, we obtain
\begin{equation}\label{nabla-e}
\begin{split}
\nabla_{e_{2}}e_{1}&=\frac{\sinh u}{\cosh^{2}u}e_{2},\
   \nabla_{e_{2}}e_{2}=-\frac{\sinh u}{\cosh^{2}u}e_{1},\\ \nabla_{e_{3}}e_{1}&=\frac{1}{u\cosh u}e_{3},\
\nabla_{e_{3}}e_{3}=-\frac{1}{u\cosh u}e_{1}.
\end{split}
\end{equation}
The structure $Q$ acts on $\{e_{i}\}$ as follows:
\begin{equation}\label{Qe}
  Qe_{1}=e_{2},\ Qe_{2}=-e_{1},\ Qe_{3}=e_{3}.
\end{equation}
Then for the structure $P=Q^{2}$ we have
\begin{equation}\label{Pe}
  Pe_{1}=-e_{1},\ Pe_{2}=-e_{2},\ Pe_{3}=e_{3}.
\end{equation}
The components of the associated metric $\tilde{g}$ are:
\begin{equation}\label{tilde-g}
\begin{array}{ll}
\tilde{g}(e_{1}, e_{1})= \tilde{g}(e_{2}, e_{2})=-1,\ \tilde{g}(e_{3}, e_{3})=1,\\
\tilde{g}(e_{1}, e_{2})= \tilde{g}(e_{2}, e_{3})=\tilde{g}(e_{1}, e_{3})=0.
\end{array}
\end{equation}
By using \eqref{F}, \eqref{nabla-e}, \eqref{Pe} and \eqref{tilde-g}
we get the components $F_{ijk}$ of $F$, $\theta_i$ of $\theta$ and  $\widetilde{\theta}_i=P_{i}^{a}\theta_{a}$ with respect to $\{e_i\}$.
The nonzero of them are as follows:
\begin{equation}\label{F313}
  F_{313}=-\frac{2}{u\cosh u},\quad
 \theta_{1}=-\frac{2u}{\cosh u},\quad
 \widetilde{\theta}_{1}=\frac{2u}{\cosh u}.
\end{equation}
Hence \eqref{g-e}, \eqref{tilde-g} and \eqref{F313} imply that the condition
\eqref{usl-w1} holds. So, we state the following
\begin{proposition}\label{kt2}
If $(S, g, P)$ is defined by \eqref{catenoid}, \eqref{g-e} and \eqref{Pe}, then the fundamental tensor $F$ on $(S, g, P)$ satisfies the property \eqref{c1}.
\end{proposition}
Next we get
\begin{proposition}\label{ro-primer}
 Let $(S, \tilde{g}, Q)$ and $(S, g, Q)$  be the manifolds defined by \eqref{catenoid}, \eqref{g-e}, \eqref{Qe} and \eqref{tilde-g}. For the Ricci tensors  $\rho$ and $\tilde{\rho}$ and for the scalar quantities $\tau$, $\tau^{*}$, $\tilde{\tau}$ and $\tilde{\tau}^{*}$ the identity \eqref{con-AE} holds.
\end{proposition}
\begin{proof}
We calculate the components $R_{ijks}$ of the curvature tensor $R$ with respect to $\{e_i\}$, having in mind \eqref{R-def}, \eqref{R-4}, \eqref{nabla-e} and the symmetries of $R$. The nonzero of them are:
\begin{align}\label{r-primer}
  R_{1212}=\frac{1}{\cosh^{4}u},\ R_{2323}=\frac{\sinh u}{u\cosh^{3}u},\ R_{1313}=-\frac{\sinh u}{u\cosh^{3}u}.
\end{align}
Using \eqref{def-rho} and \eqref{def-rho*} we compute the components $\rho_{ij}$ of $\rho$ and the values of $\tau$ and $\tau^{*}$. The nonzero of them are as follows:
\begin{equation}\label{ro}
\begin{split}
    \rho_{11}&=-\frac{1}{\cosh^{4}u}+\frac{\sinh u}{u\cosh^{3}u},\quad \rho_{22}=-\frac{1}{\cosh^{4}u}-\frac{\sinh u}{u\cosh^{3}u},\quad \rho_{33}=0,\\ \tau^{*}&=\frac{2}{\cosh^{4}u},\quad
    \tau =-\frac{2}{\cosh^{4}u}.
    \end{split}
\end{equation}
 From the Koszul formula applied to \eqref{skobki} and  \eqref{tilde-g}, we obtain
\begin{equation*}
\begin{split}
\tilde{\nabla}_{e_{2}}e_{1}&=\frac{\sinh u}{\cosh^{2}u}e_{2},\
   \tilde{\nabla}_{e_{2}}e_{2}=-\frac{\sinh u}{\cosh^{2}u}e_{1},\\ \tilde{\nabla}_{e_{3}}e_{1}&=\frac{1}{u\cosh u}e_{3},\
\tilde{\nabla}_{e_{3}}e_{3}=\frac{1}{u\cosh u}e_{1}.
\end{split}
\end{equation*}
By the above equalities we calculate the components $\tilde{R}_{ijks}$ of the curvature tensor $\tilde{R}$ with respect to $\{e_i\}$, and the nonzero of them are:
\begin{align}\label{r-tilde}
  \tilde{R}_{1212}=-\frac{1}{\cosh^{4}u},\ \tilde{R}_{2323}=\frac{\sinh u}{u\cosh^{3}u},\ \tilde{R}_{1313}=-\frac{\sinh u}{u\cosh^{3}u}.
\end{align}
Using \eqref{def-rho2}, \eqref{def-rho*} and \eqref{r-tilde}, we compute the components $\tilde{\rho}_{ij}$ of $\tilde{\rho}$ and the values of $\tilde{\tau}$ and $\tilde{\tau}^{*}$. The nonzero of them are as follows:
\begin{equation*}
\begin{split}
    \tilde{\rho}_{11}&=-\frac{1}{\cosh^{4}u}+\frac{\sinh u}{u\cosh^{3}u},\quad \tilde{\rho}_{22}=-\frac{1}{\cosh^{4}u}-\frac{\sinh u}{u\cosh^{3}u},\quad \tilde{\rho}_{33}=0,\\ \tilde{\tau}^{*}&=-\frac{2}{\cosh^{4}u},\quad
    \tilde{\tau}=\frac{2}{\cosh^{4}u}.
    \end{split}
\end{equation*}
Then, from \eqref{g-e}, \eqref{tilde-g} and \eqref{ro}, we obtain \begin{equation*}
     \tilde{\rho}_{ij} = \rho_{ij}+\frac{1}{8}(3\tilde{\tau}^{*}+\tilde{\tau}-3\tau-\tau^{*})g_{ij}+\frac{1}{8}(3\tilde{\tau}+\tilde{\tau}^{*}-3\tau^{*}-\tau)\tilde{g}_{ij},
\end{equation*}
which is a local form of \eqref{con-AE}.
\end{proof}
Now, with the help of \eqref{3.3}, \eqref{g-e}, \eqref{tilde-g}, \eqref{r-primer} and \eqref{r-tilde}, we get the following
\begin{proposition}
The sectional curvatures of the basic 2-planes of the manifolds $(S, g, Q)$ and $(S, \tilde{g}, Q)$, are
\begin{align*}
  k(e_{1},e_{2})=\frac{1}{\cosh^{4}u},\ k(e_{2},e_{3})=\frac{\sinh u}{u\cosh^{3}u},\ k(e_{1},e_{3})=-\frac{\sinh u}{u\cosh^{3}u},
\end{align*}
\begin{align*}
  \tilde{k}(e_{1},e_{2})=-\frac{1}{\cosh^{4}u},\ \tilde{k}(e_{2},e_{3})=-\frac{\sinh u}{u\cosh^{3}u},\ \tilde{k}(e_{1},e_{3})=\frac{\sinh u}{u\cosh^{3}u}.
\end{align*}
\end{proposition}
\begin{figure}[h]
\begin{center}
\includegraphics[width=0.4\paperwidth]{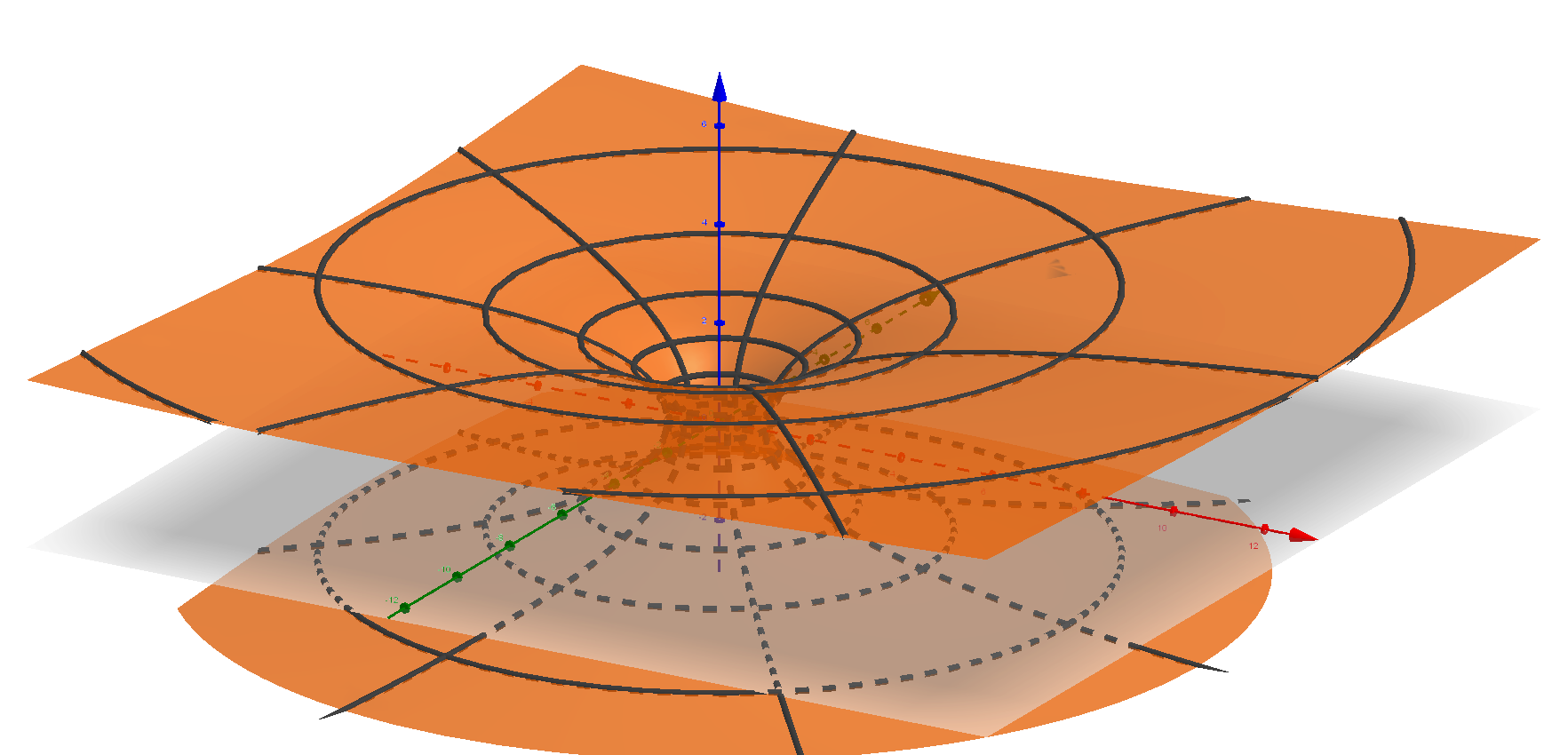}
\caption{Intersection with $Ox^{1}x^{2}x^{3}$}\label{fig1}
\end{center}
\end{figure}
\begin{figure}[h]
\begin{center}
\includegraphics[width=0.4\paperwidth]{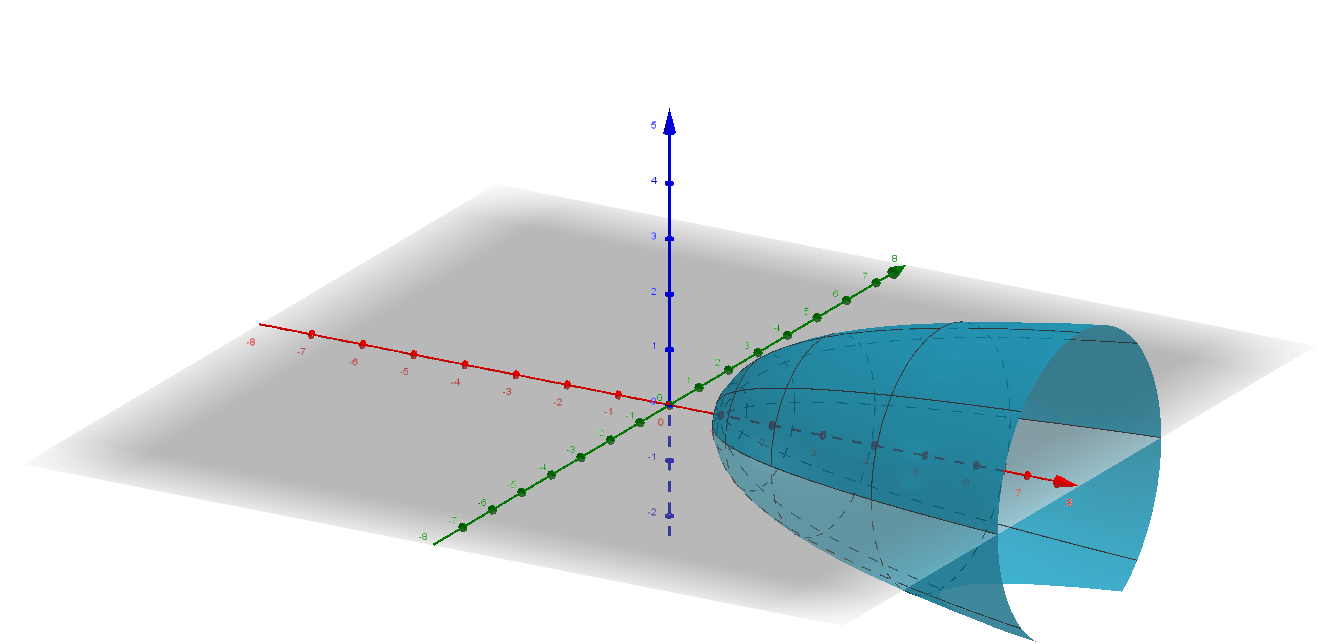}
\caption{Intersection with $Ox^{1}x^{3}x^{4}$}\label{fig2}
\end{center}
\end{figure}

\begin{remark} The intersections of $S$ with coordinate planes $Ox^{1}x^{2}x^{3}$ and $Ox^{1}x^{3}x^{4}$ are
the surfaces $S_{1}:r(\cosh u\cos v, \cosh u\sin v,u)$, $v\in [0, 2\pi)$ and $S_{2}:r(\cosh u,u\cos w,u\sin w)$,  $w\in [0, 2\pi)$, respectively. Their graphs are given in Figure~\ref{fig1} and Figure~\ref{fig2}.\end{remark}
\begin{remark} The idea for this example is inspired by papers like \cite{aydin}, \cite{erkan}, \cite{manev}.

The considered manifolds $(S, g, Q)$ and $(S, \tilde{g}, Q)$ are not Einstein or almost Einstein. Therefore, we will look for such examples in our future work.
\end{remark}


\end{document}